\documentclass[12pt,english,reqno]{amsart}

\usepackage{amsthm,amsmath,amssymb,amsfonts}
\usepackage{mathtools,mleftright}
\usepackage[mathscr]{euscript}
\usepackage{a4wide}

\usepackage[
  pdfauthor={},
  pdfkeywords={},
  pdftitle={},
  pdfcreator={},
  pdfproducer={},
  linktocpage,colorlinks,bookmarksnumbered,linkcolor=blue,
  citecolor=red,urlcolor=red]{hyperref}

\theoremstyle{plain}
\newtheorem{theorem}{Theorem}
\newtheorem*{theorem*}{Theorem}
\newtheorem{lemma}[theorem]{Lemma}
\newtheorem{corollary}[theorem]{Corollary}

\theoremstyle{definition}
\newtheorem{example}[theorem]{Example}
\newtheorem*{remark}{Remark}

\numberwithin{equation}{section}
\numberwithin{theorem}{section}

\DeclarePairedDelimiter{\floor}{\lfloor}{\rfloor}
\DeclarePairedDelimiter{\norm}{|}{|}
\DeclarePairedDelimiter{\coeff}{[}{]}

\newcommand{\set}[1]{\left\{#1\right\}}

\newcommand{\mphrase}[1]{``#1"}
\newcommand{\textdef}[1]{\textit{#1}}
\newcommand{\Case}[1]{\mbox{\textbf{Case} #1:}}
\newcommand{\andq}{\quad \text{and} \quad}

\newcommand{\ZZ}{\mathbb{Z}}

\newcommand{\CC}{\mathbb{C}}

\newcommand{\mids}{\,\mid\,}

\newcommand{\valueat}[1]{{\,}_{\big|\, #1}}
\newcommand{\dop}{\mathcal{D}}

\newcommand{\AD}{\mathrm{A}}
\newcommand{\AP}{\mathbf{A}}
\newcommand{\AR}{\mathbf{A}^{\!\!\star}}
\newcommand{\AG}{\mathcal{A}}
\newcommand{\AS}{\mathbb{A}}
\newcommand{\ASR}{\mathbb{A}^{\!\!\star}}
\newcommand{\ac}{\mathfrak{a}}
\newcommand{\an}{\mathbf{a}}

\newcommand{\BN}{\mathbf{B}}
\newcommand{\BR}{\mathbf{B}^{\star}}
\newcommand{\EN}{\mathbf{E}}
\newcommand{\ER}{\mathbf{E}^{\star}}
\newcommand{\GN}{\mathbf{G}}
\newcommand{\LN}{\mathbf{L}}

\newcommand{\FN}{\mathscr{F}\mspace{-3mu}}
\newcommand{\FB}{\mathscr{F}}
\newcommand{\FF}{\mathfrak{F}}
\newcommand{\FP}{\mathbf{F}\mspace{-3mu}}
\newcommand{\FPD}{\mathbf{F}'\mspace{-7mu}}
\newcommand{\FPDD}{\mathbf{F}''\mspace{-11mu}}
\newcommand{\fp}{\mathbf{f}\mspace{-1mu}}
\newcommand{\fh}{\widehat{\mathbf{f}}\mspace{-1mu}}
\newcommand{\dd}{\mathrm{d}}

\newcommand{\linktext}[1]{\scriptsize\texttt{#1}}
\newcommand{\arXiv}[1]{\href{https://arxiv.org/abs/#1}{\linktext{arXiv:#1~[math.NT]}}}

\title[On (self-) reciprocal Appell polynomials]
{On (self-) reciprocal Appell polynomials:\\
Symmetry and Faulhaber-type polynomials}
\author{Bernd C. Kellner}
\address{G\"ottingen, Germany}
\email{bk@bernoulli.org}

\begin{document}

\begin{abstract}
The main purpose of this paper is to study generalized (self-) reciprocal
Appell polynomials, which play a certain role in connection with Faulhaber-type
polynomials. More precisely, we show for any Appell sequence when satisfying a
reflection relation that the Appell polynomials can be described by
Faulhaber-type polynomials, which arise from a quadratic variable substitution.
Furthermore, the coefficients of the latter polynomials are given by values of
derivatives of generalized reciprocal Appell polynomials. Subsequently, we show
some applications to the Bernoulli and Euler polynomials. In the context of
power sums the results transfer to the classical Faulhaber polynomials.
\end{abstract}

\maketitle


\section{Introduction}

The Appell sequences, which were introduced by Appell~\cite{Appell:1880} in 1880,
form an important class of polynomials that are applied in numerous fields of
mathematics. Moreover, they are a subset of the wider class of Sheffer sequences,
see Roman~\cite{Roman:1984} and Rota~\cite{Rota:1975}.

Let $(\alpha_n)_{n \geq 0}$ be any sequence of numbers in $\CC$, which defines
the exponential generating functions
\begin{align*}
  \AD(t) &= \sum_{n \geq 0} \alpha_n \frac{t^n}{n!} \\
\shortintertext{and}
  \AD(t) e^{xt} &= \sum_{n \geq 0} \AP_n(x) \frac{t^n}{n!}, \\
\shortintertext{where $(\AP_n(x))_{n \geq 0}$ is called an Appell sequence.
Its members are the Appell polynomials, which satisfy for $n \geq 0$ that}
  \AP_n(x) &= \sum_{\nu=0}^{n} \binom{n}{\nu} \alpha_{n-\nu} \, x^\nu \\
\shortintertext{and}
  \AP_n(0) &= \alpha_n.
\end{align*}

Usually, it is required that $\alpha_0 \neq 0$ to ensure that $\deg \AP_n = n$.
However, we do not need this restriction here.
The Appell polynomials obey the equivalent properties
\begin{alignat}{2}
  \AP'_n(x) &= n \AP_{n-1}(x) &\quad &(n \geq 1), \nonumber \\
  \AP_n^{(\nu)}(x) &= (n)_\nu \, \AP_{n-\nu}(x) &\quad &(n \geq \nu \geq 1),
    \label{eq:ap-deriv} \\
  \AP_n(x+y) &= \sum_{\nu=0}^{n} \binom{n}{\nu} \AP_{n-\nu}(x) \, y^\nu
    &\quad &(n \geq 0), \label{eq:ap-trans}
\end{alignat}
where $(n)_\nu$ denotes the falling factorial.

We consider further results, if the reflection relation
\begin{equation*} \label{eq:ap-ref}
  \AP_n(1-x) = (-1)^n \AP_n(x) \tag{R}
\end{equation*}
additionally holds. Property \eqref{eq:ap-ref} immediately implies that
\begin{alignat*}{2}
  \AP_n(1) &= (-1)^n \alpha_n &\quad &(n \geq 0),
    \label{eq:ap-ref-2} \tag{R$_2$} \\
  \AP_n(\tfrac{1}{2}) &= 0 &\quad &(n \geq 1 \text{ odd}).
    \label{eq:ap-ref-3} \tag{R$_3$}
\end{alignat*}
The apparently weaker properties \eqref{eq:ap-ref-2} and \eqref{eq:ap-ref-3}
actually hold as an equivalence to \eqref{eq:ap-ref}. Indeed, this easily
follows from the translation formula \eqref{eq:ap-trans}. (For the case
\eqref{eq:ap-ref-3}, see Lemma~\ref{lem:ap-ref-h} in Section~\ref{sec:ft-poly}
for a more general result.) Considering \eqref{eq:ap-ref-2} in the case
\mbox{$n=1$}, the first coefficients satisfy that
\begin{equation*} \label{eq:ap-ref-4}
  \alpha_1 = - \tfrac{1}{2} \alpha_0. \tag{R$_4$}
\end{equation*}

For a polynomial $P_n(x) = c_n x^n + \dotsb + c_1 x + c_0$ with generic
coefficients, its \textdef{reciprocal polynomial} is defined by
\[
  P^\star_n(x) = x^n P_n(x^{-1}).
\]
If $P_n(x) = P^\star_n(x)$, then $P_n$ is
called \textdef{self-reciprocal} or \textdef{palindromic}, since the
coefficients of $P_n$ form a palindromic sequence such that $c_\nu = c_{n-\nu}$
for $0 \leq \nu \leq n$. Alternatively, if the condition $c_\nu = -c_{n-\nu}$
(respectively, $\norm{c_\nu} = \norm{c_{n-\nu}}$) holds, then $P_n$ is called
\textdef{anti-palindromic} (respectively, \textdef{quasi-palindromic}).
Generally, it is recommended that $n = \deg P_n$. The definitions here
do not depend on the degree of $P_n$ but rather on the index $n$, since it can
happen that $c_n = 0$ or $c_0 = 0$ when evaluating the generic coefficients,
so $n \neq \deg P_n$ or $n \neq \deg P^\star_n$ in that case, respectively.

The reciprocal Appell polynomials for $n \geq 0$ are given by
\begin{equation} \label{eq:ar-def}
  \AR_n(x) = \sum_{\nu=0}^{n} \binom{n}{\nu} \alpha_\nu \, x^\nu,
\end{equation}
where special values are
\[
  \AR_n(-1) = (-1)^n \AP_n(-1), \quad \AR_n(0) = \alpha_0,
    \andq \AR_n(1) = \AP_n(1).
\]
Moreover, if property \eqref{eq:ap-ref} holds, then we have equivalently that
\begin{equation*} \label{eq:ar-ref}
  \AR_n(1) = (-1)^n \alpha_n \quad (n \geq 0) \andq
  \AR_n(2) = 0 \quad (n \geq 1 \text{ odd}). \tag{R$_5$}
\end{equation*}

For example, the Bernoulli and Euler polynomials $\BN_n(x)$ and $\EN_n(x)$,
respectively, are Appell polynomials with $\alpha_0 = 1$ and satisfy the
reflection relation~\eqref{eq:ap-ref}. Note that $\BN_n = \BN_n(0)$ and
$\EN_n = \ER_n(2)$ are the Bernoulli and Euler numbers, respectively.
Further note that $\BR_n(2) = \BN_n + \frac{1}{2} \GN_n$ and
$\EN_n(0) = \GN_{n+1}/(n+1)$ with the Genocchi numbers $\GN_n = 2( 1 - 2^n ) \BN_n$.
See, e.g., Comtet \cite[Sec.~1.14, pp.~48--49]{Comtet:1974} and
N{\o}rlund \cite[Chap.~2, pp.~17--29]{Norlund:1924}.
Since $\BN_n(x)$ and $\EN_n(x)$ satisfy~\eqref{eq:ap-ref}, property~\eqref{eq:ar-ref}
immediately implies the well-known facts from another point of view that
$\EN_n = 0$ and $\BN_n + \frac{1}{2} \GN_n = 0$ for odd $n \geq 1$.
Due to \mbox{$\BN_1 = -\frac{1}{2}$}, it finally follows that $\BN_n = 0$
for odd $n \geq 3$.

The main aim of the paper is to find relations of the form
\[
  \AP_n(x) = \psi_n(u) \, \FP_n(u),
\]
where $\FP_n(u)$ is a Faulhaber-type polynomial with the quadratic substitution
$u = x(x-1)$ and $\psi_n(u)$ is a suitable factor. We show that this can be
accomplished whenever the reflection relation~\eqref{eq:ap-ref} holds.
Moreover, it turns out that the coefficients of $\FP_n(u)$ are connected in
some way with the reciprocal Appell polynomials $\AR_n(x)$.

We have recently studied in \cite{Kellner:2021} the special case of reciprocal
Bernoulli polynomials in the context of power sums
$S_n(m) = \sum_{\nu=0}^{m-1} \nu^n$ \mbox{$(m, n \geq 0)$}
and the original Faulhaber polynomials $\FF_n(y)$.
About 400 years ago, Johann Faulhaber (1580--1635) first discovered,
among many other relations and without knowledge of the Bernoulli numbers and
polynomials, the remarkable equations
\[
  S_n(m) = \FF_n(y)
\]
for odd $n \in \set{5,7,\ldots,17}$ where \mbox{$y = S_1(m) = m(m-1)/2$},
thereby extending known identities for $n \in \set{1,3}$ from ancient times.
(The case $n$ even needs an extra factor and is omitted here,
see~\cite{Kellner:2021} for historical notes.) Since
\[
  S_n(m) = \int_{0}^{m} \BN_n(x) \, dx,
\]
the connection with Appell polynomials suggests similar but more generalized
results. The case of power sums will then be revisited in the final section
and follow as an application of these results, see Example~\ref{exp:appl}.

The rest of the paper is organized as follows. The next section discusses the
properties of reciprocal and palindromic Appell polynomials.
Section~\ref{sec:umbral} is devoted to umbral calculus.
In Section~\ref{sec:ft-poly} we introduce the Faulhaber-type polynomials by
showing several properties. The final Section~\ref{sec:sym} establishes the
relationship between Appell and Faulhaber-type polynomials with the help of
generalized reciprocal Appell polynomials. This finally results in
Theorem~\ref{thm:fp-main}, giving the main statements at the end following
several notation and definitions.


\section{Reciprocal and palindromic polynomials}

Define the differential operator
\[
  \dop_t = \frac{d}{dt}.
\]
Let $\coeff{t^n}$ be the linear functional that gives the coefficient of $t^n$
of a formal power series such that
\[
  f(t) = \sum_{n \geq 0} a_n t^n, \quad \coeff{t^n} \, f(t) = a_n.
\]
We extend the definition of a reciprocal polynomial in case of ambiguity.
If we have a polynomial $P_{n,k}(x)$ that depends on two indices, then we
write explicitly, e.g.,
\[
  P_{n,k}(x) = (-1)^n x^k P_{n,k}(x^{-1}) = (-1)^n P^\star_{n,k}(x).
\]

The signed Lah numbers for $n \geq k \geq 1$ are defined by
\[
  \LN_{n,k} = (-1)^n \frac{n!}{k!} \binom{n-1}{k-1},
\]
which were introduced by Lah~\cite{Lah:1954} in 1954
(cf.~Comtet~\cite[Ex.\,2, p.\,156]{Comtet:1974}).
As an application, we have the following lemma for derivatives of composite
functions of a special type.

\begin{lemma}
Let $f \in C^\infty$. For $n \geq 1$, we have
\[
  \dop_x^n \, f(x^{-1}) = \sum_{k=1}^{n} \LN_{n,k} \, x^{-(n+k)} f^{(k)}(x^{-1}).
\]
\end{lemma}

This Lemma is given as an exercise for the use of Fa\`a di Bruno's formula
in \cite[Ex.\,7, p.\,157]{Comtet:1974}. For a short proof using Hoppe's
formula, see \cite{Kellner:2021}. Instantly, we obtain the following result.

\begin{lemma}
For $n \geq \ell \geq 1$, we have
\begin{align*}
  \dop_x^\ell \, \AP_n(x^{-1}) &= \sum_{\nu=1}^{\ell} \LN_{\ell,\nu} \,
    x^{-(\ell+\nu)} (n)_\nu \, \AP_{n-\nu}(x^{-1}).
\end{align*}
\end{lemma}

To go further, we define the \textdef{generalized} reciprocal Appell polynomials as
\begin{equation} \label{eq:ag-def}
  \AG_{n,k}(x) = x^k \AP_n(x^{-1}) = x^{k-n} \AR_n(x)
    \quad (n \geq 0, \, k \in \ZZ),
\end{equation}
where
\begin{align}
  \AG_{n,k}(0) &= \begin{cases}
    \infty, & \text{if } k < n, \\
    \alpha_0, & \text{if } k = n, \\
    0, & \text{if } k > n.
  \end{cases} \nonumber
\shortintertext{If the reflection relation \eqref{eq:ap-ref} holds, then}
  \AG_{n,k}(1) &= (-1)^n \alpha_n \quad (n \geq 0, \, k \in \ZZ).
    \label{eq:ag-val-1}
\end{align}

\begin{theorem} \label{thm:ag-deriv}
For $n, k \geq \ell \geq 0$, we have
\begin{align}
  \dop_x^\ell \, \AG_{n,k}(x) &= \ell! \sum_{\nu=0}^{\ell} (-1)^\nu
    \binom{k-\nu}{k-\ell} \binom{n}{\nu} \AG_{n-\nu,k-\ell-\nu}(x).
  \label{eq:ag-deriv} \\
\shortintertext{In particular,}
  \AG'_{n,k}(x) &= k \AG_{n,k-1}(x) - n \AG_{n-1,k-2}(x). \nonumber \\
\shortintertext{If property \eqref{eq:ap-ref} holds, then}
  \AG_{n,k}^{(\ell)}(1) &= (-1)^n \ell! \sum_{\nu=0}^{\ell}
    \binom{k-\nu}{k-\ell} \binom{n}{\nu} \alpha_{n-\nu}.
  \label{eq:ag-deriv-1}
\end{align}
\end{theorem}

\begin{proof}
We first consider \eqref{eq:ag-deriv}, where the case \mbox{$\ell=0$} is trivial.
The simple rule of the first derivative follows from
\[
  \AG'_{n,k}(x) = k x^{k-1} \AP_n(x^{-1}) - n x^k \AP_{n-1}(x^{-1}) / x^2,
\]
which coincides with \eqref{eq:ag-deriv} in the case $\ell = 1$.
Now we use induction on $\ell$. Let $n, k \geq \ell \geq 2$.
We assume that \eqref{eq:ag-deriv} holds for $\ell-1$ and prove for $\ell$.
The derivative of \eqref{eq:ag-deriv} in the case $\ell-1$ provides
\begin{align*}
  \dop_x^\ell \, \AG_{n,k}(x) = \; &(\ell-1)! \sum_{\nu=0}^{\ell-1} (-1)^\nu
    \binom{k-\nu}{k-\ell+1} \binom{n}{\nu} \; \times \\
  & \big( (k-\ell-\nu+1) \AG_{n-\nu,k-\ell-\nu}(x)
    -(n-\nu) \AG_{n-\nu-1,k-\ell-\nu-1}(x) \big).
\end{align*}

We split the sum into two parts as follows:
\[
  \frac{1}{(\ell-1)!} \dop_x^\ell \, \AG_{n,k}(x) = S_1 + S_2,
\]
where
\begin{align*}
  S_1 &= \sum_{\nu=0}^{\ell-1} (-1)^\nu
    \binom{k-\nu}{k-\ell+1} \binom{n}{\nu} (k-\ell-\nu+1)
    \AG_{n-\nu,k-\ell-\nu}(x), \\
\shortintertext{and by shifting the index}
  S_2 &= \sum_{\nu=1}^{\ell} (-1)^\nu
    \binom{k-\nu+1}{k-\ell+1} \binom{n}{\nu-1} (n-\nu+1)
    \AG_{n-\nu,k-\ell-\nu}(x).
\end{align*}
Note that $\binom{n}{\nu-1} (n-\nu+1) = \binom{n}{\nu} \nu$.
We evaluate the coefficients of $S_1+S_2$ in the following cases.
Comparing with \eqref{eq:ag-deriv}, we can neglect the coincident factors
$(\ell-1)!$ and $(-1)^\nu \binom{n}{\nu}$.

\Case{$\nu=0$} We have $c_0 = \binom{k}{k-\ell+1} (k-\ell+1) = \binom{k}{k-\ell} \ell$.

\Case{$\nu=\ell$} We have $c_\ell = \binom{k-\ell+1}{k-\ell+1} \ell = \ell$.

\Case{$\nu=1,\ldots,\ell-1$} It follows that
\begin{align*}
 c_\nu &= \binom{k-\nu}{k-\ell+1} (k-\ell-\nu+1) + \binom{k-\nu+1}{k-\ell+1} \nu \\
   &= \binom{k-\nu}{k-\ell} \frac{(\ell-\nu)(k-\ell-\nu+1)+(k-\nu+1)\nu}{k-\ell+1}
   = \binom{k-\nu}{k-\ell} \ell.
\end{align*}

Together with the omitted factors, the coefficients $c_\nu$ coincide with
\eqref{eq:ag-deriv} in all cases. This shows the induction. Additionally, if
property \eqref{eq:ap-ref} holds, then formula \eqref{eq:ag-deriv-1} follows
from \eqref{eq:ag-deriv} by using \eqref{eq:ag-val-1}. This completes the proof.
\end{proof}

A direct proof of Theorem~\ref{thm:ag-deriv} without using induction was given
in \cite{Kellner:2021} in the context of Bernoulli polynomials,
which shows the origin of formula \eqref{eq:ag-deriv}.
By \eqref{eq:ag-def}, we obtain a similar result for $\AR_n(x)$ as follows.

\begin{corollary} \label{cor:ar-deriv}
For $n \geq \ell \geq 0$, we have
\begin{align*}
  \frac{x^\ell}{\ell!} \, \dop_x^\ell \, \AR_n(x)
    &= \sum_{\nu=0}^{\ell} (-1)^\nu
    \binom{n-\nu}{n-\ell} \binom{n}{\nu} \AR_{n-\nu}(x).
\shortintertext{In particular,}
  x \AR_n(x)' &= n \, ( \AR_n(x) - \AR_{n-1}(x) ).
\shortintertext{If property \eqref{eq:ap-ref} holds, then}
  {\AR_n}^{\!\!(\ell)}(1) &= (-1)^n \ell! \sum_{\nu=0}^{\ell}
    \binom{n-\nu}{n-\ell} \binom{n}{\nu} \alpha_{n-\nu}.
\end{align*}
\end{corollary}

Define for $n \geq 0$ the related functions to $\AR_n(x)$ by
\[
  \AS_n(x) = \sum_{k=0}^{n} \ac_{n,k} \, x^k
\]
with the coefficients
\[
  \ac_{n,k} = \sum_{\nu = k}^{n} \binom{n}{\nu} \binom{\nu}{k} \alpha_\nu
    \quad (0 \leq k \leq n).
\]

The reflection relation \eqref{eq:ap-ref} implies that self-reciprocal
polynomials exist.

\begin{theorem} \label{thm:as-palin}
For $n \geq 0$, we have
\begin{align*}
  \AS_n(x) &= \AR_n(x+1). \\
\shortintertext{If property \eqref{eq:ap-ref} holds, then
the polynomials $\AS_n(x)$ are quasi-palindromic such that}
  \AS_n(x) &= (-1)^n \ASR_n(x)
\shortintertext{and for $0 \leq k \leq n$ that}
  \ac_{n,k} &= (-1)^n \ac_{n,n-k}.
\end{align*}
\end{theorem}

\begin{proof}
Let $n \geq 0$ and $0 \leq k \leq n$. From \eqref{eq:ar-def} we derive that
\[
  \frac{1}{k!} {\AR_n}^{\!\!(k)}(1) = \sum_{\nu = k}^{n} \binom{n}{\nu}
    \binom{\nu}{k} \alpha_\nu = \ac_{n,k},
\]
which implies that $\AS_n(x) = \AR_n(x+1)$. Now, we assume that property
\eqref{eq:ap-ref} holds. From Corollary~\ref{cor:ar-deriv} and relations
\eqref{eq:ag-def} and \eqref{eq:ag-val-1}, we then infer that
\[
  \frac{1}{k!} {\AR_n}^{\!\!(k)}(1) = (-1)^n \sum_{\nu=0}^{k}
    \binom{n}{\nu} \binom{n-\nu}{n-k} \alpha_{n-\nu}
    = (-1)^n \ac_{n,n-k}.
\]
Thus, $\ac_{n,k} = (-1)^n \ac_{n,n-k}$ for $0 \leq k \leq n$, which implies
that $\AS_n(x) = (-1)^n \ASR_n(x)$.
\end{proof}

To generalize the results, we define the formal power series
\begin{equation} \label{eq:as-gen}
  \AS_{n,k}(x) = \AS_n(x) \, (x+1)^{k-n}
    = \AS_n(x) \sum_{\nu \geq 0} \binom{k-n}{\nu} x^\nu
    \quad (n \geq 0, \, k \in \ZZ).
\end{equation}

\begin{theorem} \label{thm:as-palin-2}
For $n \geq 0$ and $k \in \ZZ$, we have
\[
  \AS_{n,k}(x) = \AG_{n,k}(x+1),
\]
where we require \mbox{$\norm{x} < 1$} for convergence in case \mbox{$k < n$}.
If \mbox{$k \geq n$}, then $\AS_{n,k}(x)$ is a poly\-nomial.
In particular, we have for $n, k \geq \ell \geq 0$ that
\begin{equation} \label{eq:as-coeff}
  \coeff{x^\ell} \, \AS_{n,k}(x) = \frac{1}{\ell!} \AG_{n,k}^{(\ell)}(1)
    = \sum_{\nu=0}^{\ell} \binom{k-n}{\ell-\nu} \ac_{n,\nu}.
\end{equation}

If property~\eqref{eq:ap-ref} holds, then $\AS_{n,k}(x)$ is for
$k \geq n \geq 0$ quasi-palindromic such that
\[
  \AS_{n,k}(x) = (-1)^n x^k \AS_{n,k}(x^{-1}) = (-1)^n \ASR_{n,k}(x).
\]
In case $n$ odd and $k$ even, the central coefficient of $\AS_{n,k}(x)$
vanishes such that
\[
  \coeff{x^{k/2}} \, \AS_{n,k}(x) = 0.
\]
\end{theorem}

\begin{proof}
Let $n \geq 0$ and $k \in \ZZ$. Set $c = k-n$. By Theorem~\ref{thm:as-palin}
and definitions \eqref{eq:ag-def} and \eqref{eq:as-gen}, we have
$\AS_{n,k}(x) = \AG_{n,k}(x+1)$. If $c < 0$, then \eqref{eq:as-gen} only holds
for $\norm{x} < 1$, while $\AS_n(x) (x+1)^c$ is a polynomial for $c \geq 0$.

Next we show \eqref{eq:as-coeff}. Let $n, k \geq \ell \geq 0$.
By Leibniz's rule we obtain
\[
  \dop_x^\ell \, \AS_{n,k}(x+1) = \sum_{\nu=0}^{\ell} \binom{\ell}{\nu}
    \big( \AS_n(x) \big)^{(\nu)} \big( (x+1)^c \big)^{(\ell-\nu)}.
\]
We evaluate the equation at $x=0$, which implies \eqref{eq:as-coeff}.

Now, we assume property~\eqref{eq:ap-ref}. Let $k \geq n \geq 0$, so $c \geq 0$.
By Theorem~\ref{thm:as-palin}
$\AS_n(x)$ is quasi-palindromic, so $\AS_{n,k}(x) = \AS_n(x) \, (x+1)^c$ as
a product is also quasi-palindromic, since $(x+1)^c$ is palindromic.
More precisely,
\begin{align*}
  \AS_{n,k}(x) = \AS_n(x) (x+1)^c &= (-1)^n x^n \AS_n(x^{-1}) \, x^c (x^{-1}+1)^c \\
    &= (-1)^n x^k \AS_{n,k}(x^{-1}) = (-1)^n \ASR_{n,k}(x).
\end{align*}
If $n$ is odd and $k$ is even, then
$\coeff{x^{k/2}} \, \AS_{n,k}(x) = - \coeff{x^{k/2}} \, \AS_{n,k}(x)$,
hence the central coefficient vanishes. This completes the proof.
\end{proof}


\section{Umbral calculus}
\label{sec:umbral}

We use the umbral notation $\mathfrak{u}^n = \mathfrak{u}_n$ for the symbols
$\AP$ and $\alpha$. The formulas of the introduction then turn into the
following relations:
\begin{align*}
  \AD(t) &= e^{\alpha t}, \\
  e^{(x+\alpha) t} &= \sum_{n \geq 0} \AP_n(x) \frac{t^n}{n!}, \\
  \AP_n(x) &= (x + \alpha)^n, \\
  \AR_n(x) &= (\alpha x + 1)^n.
\end{align*}

Define the polynomials
\begin{equation} \label{eq:ap-rs}
  \AP_{r,s}(x) = (\AP(x)+1)^r \AP(x)^s
    = \sum_{\nu=0}^{r} \binom{r}{\nu} \AP_{s+\nu}(x)
\end{equation}
with rank $r \geq 0$ and shift $s \geq 0$. The recurrence of the binomial
coefficients implies that
\begin{align*}
  \AP_{r+1,s}(x) &= \AP_{r,s}(x) + \AP_{r,s+1}(x). \\
\intertext{By induction this can be generalized to the relations}
  \AP_{r+n,s}(x) &= \sum_{\nu=0}^{n} \binom{n}{\nu} \AP_{r,s+\nu}(x), \\
  \AP_{r,s+n}(x) &= \sum_{\nu=0}^{n} \binom{n}{\nu} (-1)^{n-\nu} \AP_{r+\nu,s}(x).
\end{align*}

\begin{theorem}
If property \eqref{eq:ap-ref} holds, then we have for $r,s \geq 0$
the reciprocity relation
\begin{equation} \label{eq:ap-recip}
  (-1)^r \AP_{r,s}(x) = (-1)^s \AP_{s,r}(-x).
\end{equation}
\end{theorem}

\begin{proof}
From \eqref{eq:ap-trans} and \eqref{eq:ap-ref}, we infer that
\[
  \sum_{\nu=0}^{n} \binom{n}{\nu} \AP_\nu(x) = \AP_n(x+1) = (-1)^n \, \AP_n(-x).
\]
Hence, $(\AP(x)+1)^n = (-\AP(-x))^n$. By linearity we can apply the polynomial
$f(z) = z^r (z-1)^s$ to derive
\[
  (\AP(x)+1)^r \, \AP(x)^s = (-1)^{r+s} \, \AP(-x)^r \, (\AP(-x)+1)^s,
\]
which implies the result.
\end{proof}

Taking $x=0$, the above relations consequently hold for the numbers
\begin{equation} \label{eq:ac-def}
  \alpha_{r,s} = (\alpha+1)^r \alpha^s
    = \sum_{\nu=0}^{r} \binom{r}{\nu} \alpha_{s+\nu}.
\end{equation}

\begin{corollary} \label{cor:ac-recip}
If property \eqref{eq:ap-ref} holds, then we have for $r,s \geq 0$
the reciprocity relation
\begin{equation} \label{eq:ac-recip}
  (-1)^r \alpha_{r,s} = (-1)^s \alpha_{s,r}.
\end{equation}
\end{corollary}

\begin{remark}
In the case $\alpha_n = \BN_n$ with the Bernoulli numbers, this leads to a
well-known reciprocity relation, which goes back to
Ettingshausen~\cite{Ettingshausen:1827} in 1827, see \cite{Kellner:2021b}.
The application of the umbral trick with $z^r (z-1)^s$ was already used by
Lucas~\cite[Sec.\,135, p.\,240]{Lucas:1891} in 1891, who also showed many
other examples.
\end{remark}

Theorem~\ref{thm:as-palin} can be given now by a similar representation using
\eqref{eq:ac-def} and \eqref{eq:ac-recip}.

\begin{corollary} \label{cor:ac-as}
For $n \geq 0$, we have
\begin{align*}
  \AS_n(x) &= (\alpha+1+\alpha x)^n
    = \sum_{\nu=0}^{n} \binom{n}{\nu} \alpha_{\nu,n-\nu} \, x^{n-\nu}, \\
\shortintertext{where}
  \ac_{n,\nu} &= \binom{n}{\nu} \alpha_{n-\nu,\nu} \mspace{7mu}
    \quad (0 \leq \nu \leq n). \\
\shortintertext{If property \eqref{eq:ap-ref} holds, then}
  \AS_n(x) &= (-1)^n \ASR_n(x) \\
\shortintertext{and}
  \ac_{n,\nu} &= (-1)^n \ac_{n,n-\nu} \quad (0 \leq \nu \leq n).
\end{align*}
\end{corollary}

To consider the general case of \eqref{eq:ap-rs} and \eqref{eq:ap-recip},
we define for $r,s \geq 0$ the bivariate polynomials
\[
  \AP_{r,s}(x,y) = (\AP(x)+y)^r \AP(x)^s.
\]

\begin{theorem}
Let $r,s \geq 0$. Then we have the following reciprocity relations
\begin{align}
  \AP_{r,s}(x,y) &= \AP_{s,r}(x+y,-y) \label{eq:ap-biv-1} \\
\shortintertext{and assuming property \eqref{eq:ap-ref} that}
  (-1)^r \AP_{r,s}(x,y) &= (-1)^s \AP_{s,r}(1-x-y,y). \label{eq:ap-biv-2}
\end{align}
\end{theorem}

\begin{proof}
Since \mbox{$(\AP(x)+y)^n = \AP(x+y)^n$} by \eqref{eq:ap-trans}, relation
\eqref{eq:ap-biv-1} easily follows by applying the polynomial $f(z) = z^r (z-y)^s$.
Now we assume property \eqref{eq:ap-ref}. We then obtain
\begin{align*}
  \AP_{r,s}(1-x,-y) &= (\AP(1-x)-y)^r \AP(1-x)^s \\
    &= (-1)^{r+s} (\AP(x)+y)^r \AP(x)^s = (-1)^{r+s} \AP_{r,s}(x,y).
\end{align*}
Applying \eqref{eq:ap-biv-1} to $\AP_{r,s}(1-x,-y)$ finally implies
\eqref{eq:ap-biv-2}. This completes the proof.
\end{proof}

\begin{remark}
Identity \eqref{eq:ap-biv-1} is equivalent to the similar umbral relation
\[
  \mphrase{(A_k(x)+y)^m = (A_m(x+y)-y)^k \quad (k,m \geq 0)}
\]
of Agoh~\cite[Thm.~1.1]{Agoh:2021}, who derived this result recently by a
rather lengthy and complicated proof, compared to the one-line proof of
\eqref{eq:ap-biv-1} here.
\end{remark}


\section{Faulhaber-type polynomials}
\label{sec:ft-poly}

We consider the following substitution and relations
\[
  u = x(x-1), \quad u' = 2x-1, \andq (u')^2 = 4u+1,
\]
regarding the derivative with respect to~$x$. Define for $n \geq 0$ the symbols
\[
  \dd_n = \floor*{\frac{n}{2}} \andq \delta_n = \begin{cases}
    1, & \text{if $n$ is odd}, \\
    0, & \text{otherwise},
  \end{cases}
\]
satisfying the recurrence $\dd_{n+1} = \dd_n + \delta_n$. In the context of the
Appell polynomials $\AP_n(x)$, define for $n \geq 0$ the Faulhaber-type
polynomials
\[
  \FP_n(u) = \sum_{k \geq 0} \fp_{n,k} \, u^k
    = \sum_{k \geq 0} \fh_{n,k} \, \frac{u^k}{k!},
\]
where $\fp_{n,k} = \fh_{n,k} = 0$ for $k > \dd_n$. In case of ambiguity, we use
the notation, e.g., $\fp_{n,k} \valueat{\AP_n(x)}$ to indicate the specific
context. The Fibonacci numbers can be defined for $n \geq 0$ by
\[
  \FN_n = \frac{\phi^n - (1-\phi)^n}{\sqrt{5}}
\]
with the golden ratio $\phi = (1 + \sqrt{5})/2$, cf.~\cite[p.~299]{GKP:1994}.

\begin{lemma} \label{lem:ap-ref-h}
For $n \geq 0$, we have
\[
  \AP_n(x) = \sum_{\substack{\nu=0\\2 \mids \nu}}^{n} \binom{n}{\nu}
    \AP_{\nu}\mleft( \tfrac{1}{2} \mright) \,
    \mleft( x - \tfrac{1}{2} \mright)^{\!n-\nu},
\]
where the condition $2 \mid \nu$ holds, if property \eqref{eq:ap-ref} holds.
\end{lemma}

\begin{proof}
By the translation formula \eqref{eq:ap-trans}, we obtain
\[
  \AP_n(x) = \AP_n \mleft( x - \tfrac{1}{2} + \tfrac{1}{2} \mright)
    = \sum_{\substack{\nu=0\\2 \mids \nu}}^{n} \binom{n}{\nu}
    \AP_{\nu}\mleft( \tfrac{1}{2} \mright) \,
    \mleft( x - \tfrac{1}{2} \mright)^{\!n-\nu},
\]
where the condition $2 \mid \nu$ holds due to properties~\eqref{eq:ap-ref}
and \eqref{eq:ap-ref-3}.
\end{proof}

\begin{lemma} \label{lem:ap-subst}
For $n \geq 0$, we have
\begin{equation} \label{eq:ap-subst-1}
  \AP_n(x) = 2^{-n} \sum_{\nu=0}^{n} \binom{n}{\nu} \AS_\nu(1) \, (u')^{n-\nu}.
\end{equation}
If property \eqref{eq:ap-ref} holds, then
\begin{equation} \label{eq:ap-subst-2}
  \AP_n(x) = (u')^{\delta_n} \, 2^{-n} \sum_{\nu=0}^{\dd_n}
    \binom{n}{2\nu} \AS_{2\nu}(1) \, (4u+1)^{\dd_n-\nu}.
\end{equation}
\end{lemma}

\begin{proof}
Note that $\AS_\nu(1) = \AR_\nu(2) = 2^\nu \AP_\nu(\frac{1}{2})$ and
$(u')^{n-\nu} = 2^{n-\nu} (x-\frac{1}{2})^{n-\nu}$.
From Lemma~\ref{lem:ap-ref-h} the first claimed formula~\eqref{eq:ap-subst-1}
follows without the condition $2 \mid \nu$. Now we assume
property~\eqref{eq:ap-ref}. Using $(u')^2 = 4u+1$ and factoring $u'$ out
when $\nu$ is even, we then derive
\[
  \AP_n(x) = (u')^{\delta_n} \, 2^{-n} \sum_{\substack{\nu=0\\2 \mids \nu}}^{n}
    \binom{n}{\nu} \AS_\nu(1) \, (4u+1)^{(n-\delta_n-\nu)/2}.
\]
With $\dd_n = (n-\delta_n)/2$ and changing the indexing, this implies the
second claimed formula~\eqref{eq:ap-subst-2} and completes the proof.
\end{proof}

First relations between the Appell and Faulhaber-type polynomials are given as
follows.

\begin{theorem} \label{thm:fp-coeff}
Assume property \eqref{eq:ap-ref}. For $n \geq 0$, we have
\begin{equation} \label{eq:ap-fp}
  \AP_n(x) = (u')^{\delta_n} \, \FP_n(u),
\end{equation}
where
\begin{equation} \label{eq:fp-coeff}
  \fp_{n,k} = 2^{2k-n} \sum_{\nu=0}^{\dd_n-k} \binom{n}{2\nu}
    \binom{\dd_n-\nu}{k} \AS_{2\nu}(1)
\end{equation}
for $0 \leq k \leq \dd_n$. In particular, the constant term and the leading
coefficient satisfy that
\[
  \fp_{n,0} = (-1)^n \alpha_n \andq \fp_{n,\dd_n} = (\tfrac{1}{2})^{\delta_n} \alpha_0,
\]
respectively. The sum of the coefficients $\fp_{n,k}$ obey the relations
\[
  \AP_n(\phi) = (\sqrt{5})^{\delta_n} \, \FP_n(1)
\]
and with the Fibonacci numbers $\FN_\nu$ (in umbral sense) that
\[
  \AP_n(\FB) = \begin{cases}
    2 \FP_n(1), & \text{if $n$ is odd}, \\
    0, & \text{otherwise}.
  \end{cases}
\]
\end{theorem}

\begin{proof}
Using Lemma~\ref{lem:ap-subst}, we obtain \eqref{eq:ap-fp} and \eqref{eq:fp-coeff}
by expanding the powers $(4u+1)^{\dd_n-\nu}$ in \eqref{eq:ap-subst-2}.
With $x=1$, $u=0$, and $u'=1$, we obtain by \eqref{eq:ap-fp} that
\begin{align*}
  \fp_{n,0} &= \FP_n(0) = \AP_n(1) = (-1)^n \alpha_n. \\
\shortintertext{Further by \eqref{eq:fp-coeff}, it follows that}
  \fp_{n,\dd_n} &= 2^{2\dd_n-n} \AS_0(1) = 2^{-\delta_n} \alpha_0.
\end{align*}
Note that in the cases $n=0,1$ we have $\dd_n=0$, so $\fp_{n,0} = \fp_{n,\dd_n}$.
Both cases are valid, since $\fp_{n,0} = \fp_{n,\dd_n} = \alpha_0$ for $n=0$ and
$\fp_{n,0} = -\alpha_1 = \frac{1}{2} \alpha_0 = \fp_{n,\dd_n}$ by
\eqref{eq:ap-ref-4} for $n=1$.

Now, let $x=\phi$, so $u=1$ and $u'=\sqrt{5}$. Together with \eqref{eq:ap-ref},
we infer that
\[
  \AP_n(\phi) = (\sqrt{5})^{\delta_n} \, \FP_n(1) \andq
  \AP_n(1-\phi) = (-\sqrt{5})^{\delta_n} \, \FP_n(1).
\]
The difference \mbox{$\AP_n(\phi) - \AP_n(1-\phi)$} implies the claimed formula
for $\AP_n(\FB)$. This completes the proof.
\end{proof}

\begin{example}
In the case of Euler and Bernoulli polynomials, we obtain for $n \geq 0$
and $0 \leq k \leq \dd_n$ that
\begin{align*}
  \EN_n(x) &= (u')^{\delta_n} \, 2^{-n} \sum_{\nu=0}^{\dd_n}
    \binom{n}{2\nu} (\EN_{2\nu}) (4u+1)^{\dd_n-\nu},\\
  \fp_{n,k} \valueat{\EN_n(x)} &= 2^{2k-n} \sum_{\nu=0}^{\dd_n-k} \binom{n}{2\nu}
    \binom{\dd_n-\nu}{k} \EN_{2\nu}, \\
\shortintertext{respectively}
  \BN_n(x) &= (u')^{\delta_n} \, 2^{-n} \sum_{\nu=0}^{\dd_n}
    \binom{n}{2\nu} \mleft( \BN_{2\nu} + \tfrac{1}{2} \GN_{2\nu} \mright)
    (4u+1)^{\dd_n-\nu}, \\
  \fp_{n,k} \valueat{\BN_n(x)} &= 2^{2k-n} \sum_{\nu=0}^{\dd_n-k} \binom{n}{2\nu}
    \binom{\dd_n-\nu}{k} \mleft( \BN_{2\nu} + \tfrac{1}{2} \GN_{2\nu} \mright).
\end{align*}
\end{example}

\begin{remark}
In 1867, Schr\"oder \cite{Schroeder:1867} gave explicit expressions of $\BN_n(x)$
expanded in terms of~$\BN_{2\nu}$ and $\xi = -u$ in the context of Faulhaber
polynomials and power sums, see \cite{Kellner:2021} for a simplified approach.
In contrast here, $\AS_{2\nu}(1) = \BN_{2\nu} + \frac{1}{2} \GN_{2\nu}$ implies
the \emph{suitable term} at once. Lemma~\ref{lem:ap-ref-h} was given by
N{\o}rlund~\cite[p.~28]{Norlund:1924} for $\BN_n(x)$ in 1924.
\end{remark}

In view of Theorem~\ref{thm:fp-coeff}, the Faulhaber-type polynomials $\FP_n(u)$
satisfy only the Appell property $\FPD_n(u) = n \FP_{n-1}(u)$ in case $n \geq 2$
even as shown by the following theorem. For the recurrences of the coefficients
$\fp_{n,k}$, it is more convenient to switch to $\fh_{n,k}$.

\begin{theorem} \label{thm:fp-recur}
Assume property \eqref{eq:ap-ref}. There are the following recurrences:
\begin{align*}
  n \, \FP_{n-1}(u) &= \begin{cases}
    \FPD_n(u), &\text{if $n \geq 2$ is even}, \\
    2 \FP_n(u) + (4u+1) \FPD_n(u), &\text{if $n \geq 1$ is odd},
  \end{cases}\\
\shortintertext{and}
  (n)_2 \, \FP_{n-2}(u) &= \begin{cases}
    2 \FPD_n(u) + (4u+1) \FPDD_n(u), &\text{if $n \geq 2$ is even}, \\
    6 \FPD_n(u) + (4u+1) \FPDD_n(u), &\text{if $n \geq 3$ is odd}.
  \end{cases}
\shortintertext{The coefficients obey the following recurrences:}
  n \, \fh_{n-1,k} &= \begin{cases}
    \fh_{n,k+1}, &\text{if $n \geq 2$ is even}, \\
    (4k+2) \, \fh_{n,k} + \fh_{n,k+1}, \mspace{21.6mu}\ &\text{if $n \geq 1$ is odd},
  \end{cases}
\shortintertext{and}
  (n)_2 \, \fh_{n-2,k} &= \begin{cases}
    (4k+2) \, \fh_{n,k+1} + \fh_{n,k+2}, &\text{if $n \geq 2$ is even}, \\
    (4k+6) \, \fh_{n,k+1} + \fh_{n,k+2}, \mspace{5.4mu}\ &\text{if $n \geq 3$ is odd},
  \end{cases}
\end{align*}
for $0 \leq k \leq \dd_{n-1}$ and $0 \leq k \leq \dd_{n-2}$, respectively.
\end{theorem}

We need the following relations, which are easily derived.

\begin{lemma} \label{lem:fp-deriv}
For $n \geq 0$, there are the following derivatives:
\begin{align*}
  \dop_x \, \FP_n(u) &= u' \, \FPD_n(u), \\
  \dop_x^2 \, \FP_n(u) &= 2 \FPD_n(u) + (4u+1) \FPDD_n(u), \\
  \dop_x \, ( u' \, \FP_n(u) ) &= 2 \FP_n(u) + (4u+1) \FPD_n(u), \\
  \dop_x^2 \, ( u' \, \FP_n(u) ) &= u' (6 \FPD_n(u) + (4u+1) \FPDD_n(u)).
\end{align*}
\end{lemma}

\begin{proof}[Proof of Theorem~\ref{thm:fp-recur}]
Thanks to Lemma~\ref{lem:fp-deriv}, the expressions of $n \FP_{n-1}(u)$ and
$(n)_2 \FP_{n-2}(u)$ follow from applying $\dop^\ell_x$ to \eqref{eq:ap-fp} and
using $\AP_n^{(\ell)}(x) = (n)_\ell \, \AP_{n-\ell}(x)$ by \eqref{eq:ap-deriv}
for each $\ell = 1,2$, where the cases $n$ even and odd have to be handled
separately. The recurrences of the coefficients $\fh_{n,k}$ then are
subsequently deduced from using the identities
\[
  \dop^\ell_u \, \FP_n(u)
    = \sum_{k \geq 0} \fh_{n,k+\ell} \, \frac{u^k}{k!} \andq
  u \, \dop^\ell_u \, \FP_n(u)
    = \sum_{k \geq 0} k \, \fh_{n,k+\ell-1} \, \frac{u^k}{k!}.
\]
This proves the theorem.
\end{proof}


\section{Recurrences and symmetries}
\label{sec:sym}

We use the notation and definitions of the preceding sections implicitly.

\begin{lemma} \label{lem:la-recur}
Define for $n \geq k \geq 0$ the sum
\begin{align}
  \Lambda_{n,k}(\alpha) &= \sum_{\nu=0}^{k} \binom{n}{\nu} (2k-\nu)_k \,
    \alpha_{n-\nu}. \label{eq:la-def} \\
\shortintertext{For $n - 2 \geq k \geq 0$, we have the recurrence}
  (n)_2 \, \Lambda_{n-2,k}(\alpha) &= \Lambda_{n,k+2}(\alpha) - (4k+6)
    \Lambda_{n,k+1}(\alpha). \label{eq:la-recur}
\end{align}
\end{lemma}

\begin{proof}
We define $\binom{n}{\nu} = 0$ for negative indices $\nu$. This allows us to
establish the identity
\[
  (n)_2 \binom{n-2}{\nu-2} = (\nu)_2 \binom{n}{\nu} \quad (\nu \geq 0).
\]
Let $n - 2 \geq k \geq 0$. We obtain by using the above identity and shifting
the index $\nu$ that
\begin{equation} \label{eq:la-n-2}
  (n)_2 \, \Lambda_{n-2,k}(\alpha)
    = \sum_{\nu=0}^{k+2} \binom{n}{\nu} (\nu)_2 (2k+2-\nu)_k \, \alpha_{n-\nu}.
\end{equation}
Regarding the sum $\Lambda_{n,k+1}(\alpha)$ in \eqref{eq:la-recur}, we can
extend the summation over $\nu$ to $k+2$, since the added summand vanishes for
$\nu = k+2$ due to $(2k+2-\nu)_{k+1} = (k)_{k+1} = 0$ for $k \geq 0$.
Consequently, we can collect the coefficients of $\alpha_{n-\nu}$ in
\eqref{eq:la-recur} and \eqref{eq:la-n-2} for $0 \leq \nu \leq k+2$.
Neglecting the common factor $\binom{n}{\nu}$, \eqref{eq:la-recur} turns for
each $\nu$ into
\[
  (\nu)_2 (2k+2-\nu)_k = (2k+4-\nu)_{k+2} - (4k+6)(2k+2-\nu)_{k+1}.
\]
Writing as a difference and factoring the term $(2k+2-\nu)_k$ out, we infer that
\[
  (\nu)_2 - (2k+4-\nu)_2 + (4k+6)(k+2-\nu) = 0.
\]
This implies the result.
\end{proof}

The proof above was similarly given in \cite[Lem.~9.5]{Kellner:2021} with
polynomials $\Lambda_{n,k}(x)$ instead of $\Lambda_{n,k}(\alpha)$.
As an application, Lemma~\ref{lem:la-recur} is essential for the following theorem.

\begin{theorem} \label{thm:ag-recur}
Assume property \eqref{eq:ap-ref}.
For $n \geq k \geq 0$, we have
\begin{align}
  \AG_{n,2k}^{(k)}(1) &= (-1)^n k! \sum_{\nu=0}^{k}
    \binom{2k-\nu}{k} \binom{n}{\nu} \alpha_{n-\nu}.
  \label{eq:ag-2k} \\
\shortintertext{The numbers $\an_{n,k} = \AG_{n,2k}^{(k)}(1)$ obey for $n \geq 2$
and $0 \leq k \leq n-2$ the recurrence}
  \an_{n,k+2} &= (4k+6) \, \an_{n,k+1} + (n)_2 \, \an_{n-2,k}.
  \label{eq:ag-recur} \\
\shortintertext{For odd $n \geq 1$, the equation $\AP_n(x) = u' \, \FP_n(u)$
holds with the coefficients}
  \fh_{n,k} &= (-1)^k \, \an_{n,k} \quad (0 \leq k \leq \dd_n).
  \label{eq:fh-an}
\end{align}
\end{theorem}

\begin{proof}
Theorem~\ref{thm:ag-deriv} and \eqref{eq:ag-deriv-1} provide \eqref{eq:ag-2k}
for $n \geq k \geq 0$. The case $k=0$ of \eqref{eq:ag-2k} yields
\begin{alignat}{2}
  \an_{n,0} &= \AG_{n,0}(1) = (-1)^n \alpha_n &&\quad (n \geq 0)
  \label{eq:an-0} \\
\shortintertext{and the next numbers $\an_{n,1}$ are given by}
  \an_{n,1} &= (-1)^n ( 2 \alpha_n + n \alpha_{n-1} ) &&\quad (n \geq 1).
  \label{eq:an-1}
\end{alignat}
From Lemma~\ref{lem:la-recur} and \eqref{eq:la-def}, we infer that
$\an_{n,k} = (-1)^n \Lambda_{n,k}(\alpha)$.
Considering the parity of the sign $(-1)^n$, one sees that
recurrences~\eqref{eq:la-recur} and \eqref{eq:ag-recur} are equivalent.
Thus, Lemma~\ref{lem:la-recur} implies the validity of \eqref{eq:ag-recur}.

Next, we show the desired properties of the coefficients $\fh_{n,k}$.
By Theorem~\ref{thm:fp-coeff} we infer that
\begin{equation} \label{eq:fh-0}
  \fh_{n,0} = \fp_{n,0} = (-1)^n \alpha_n \quad (n \geq 0),
\end{equation}
while Theorem~\ref{thm:fp-recur} implies for odd $n \geq 3$ that
\begin{equation} \label{eq:fh-1}
  \fh_{n,1} = n \fh_{n-1,0} - 2 \fh_{n,0}
    = (-1)^{n+1} ( 2 \alpha_n + n \alpha_{n-1} )
\end{equation}
as well as the recurrence
\begin{equation} \label{eq:fh-k}
  \fh_{n,k+2} = -(4k+6) \, \fh_{n,k+1} + (n)_2 \, \fh_{n-2,k}.
\end{equation}

Comparing \eqref{eq:an-0} and \eqref{eq:an-1} with \eqref{eq:fh-0} and
\eqref{eq:fh-1}, we obtain
\begin{equation} \label{eq:fh-an-0}
  \fh_{n,0} = \an_{n,0} \quad (n \geq 1 \text{ odd}) \andq
  \fh_{n,1} = -\an_{n,1} \quad (n \geq 3 \text{ odd}).
\end{equation}
We have to show that \eqref{eq:fh-an} holds for odd $n \geq 1$ and
$0 \leq k \leq \dd_n$. By \eqref{eq:fh-an-0} this is satisfied for the initial
cases $n=1$ and $n=3$. Now, we use induction on $n$. Let $n \geq 5$ be odd.
We assume that \eqref{eq:fh-an} holds for $n-2$ and prove for $n$.
Note that $(-1)^k \, \an_{n,k}$ and $\fh_{n,k}$ satisfy equivalent recurrences
by comparing \eqref{eq:ag-recur} with \eqref{eq:fh-k} and considering the sign
change by $(-1)^k$. The initial values agree for $k=0,1$ and the cases $n$ and
$n-2$ by \eqref{eq:fh-an-0}. Using both recurrences \eqref{eq:ag-recur} and
\eqref{eq:fh-k} iteratively for $k=0,\ldots,\dd_n-2$ implies \eqref{eq:fh-an}
completely for the case $n$. This shows the induction and completes the proof.
\end{proof}

Recall Theorems~\ref{thm:as-palin} and \ref{thm:as-palin-2}. Assuming property
\eqref{eq:ap-ref}, the polynomials $\AS_n(x)$ and their coefficients
$\ac_{n,k}$ are quasi-palindromic. This also transfers to the polynomials
$\AS_{n,k}(x)$ for \mbox{$k \geq n \geq 0$}. As a main result, we achieve the
following theorem that summarizes the properties of Faulhaber-type polynomials
and their coefficients.

\begin{theorem} \label{thm:fp-main}
Assume property \eqref{eq:ap-ref}. Let $n \geq 1$ be odd.
The Faulhaber-type polynomial satisfies
\[
  \AP_n(x) = u' \, \FP_n(u),
\]
where
\begin{equation} \label{eq:fp-main}
  (-1)^k \, \fp_{n,k} = \coeff{x^k} \, \AS_{n,2k}(x)
    = \frac{1}{k!} \, \AG_{n,2k}^{(k)}(1) \quad (k \geq 0).
\end{equation}
In particular,
\[
  \AS_{n,2k}(x) \in \begin{cases}
    \CC[[x]], & \text{if $0 \leq k \leq \dd_n$}, \\
    \CC[x], & \text{if $k > \dd_n$},
  \end{cases}
\]
where in the latter case $\AS_{n,2k}(x)$ is anti-palindromic.

The coefficients satisfy that
\[
  \fp_{n,0} = - \alpha_n, \quad \fp_{n,\dd_n} = \tfrac{1}{2} \alpha_0,
  \andq \fp_{n,k} = 0 \quad (k > \dd_n).
\]
Moreover, we have for $0 \leq k \leq n$ the expressions
\begin{equation} \label{eq:fp-expr}
  \fp_{n,k} = (-1)^k \sum_{\nu=0}^{k} \binom{2k-n}{k-\nu} \ac_{n,\nu}
    = (-1)^{k+1} \sum_{\nu=0}^{k}
    \binom{2k-\nu}{k} \binom{n}{\nu} \alpha_{n-\nu}.
\end{equation}

In the even case $n+1$, we have the relation
\[
  \AP_{n+1}(x) = \FP_{n+1}(u)
\]
with
\[
  \fp_{n+1,k} = \begin{cases}
    \alpha_{n+1}, &\text{if $k=0$}, \\
    \frac{n+1}{k} \, \fp_{n,k-1}, &\text{if $1 \leq k < \dd_{n+1}$}, \\
    \alpha_0, &\text{if $k = \dd_{n+1}$}, \\
    0, &\text{if $k > \dd_{n+1}$}.
  \end{cases}
\]
\end{theorem}

\begin{proof}
Let $n \geq 1$ be odd. We first consider the restriction $0 \leq k \leq \dd_n$.
Under this restriction, the expressions \eqref{eq:fp-main} and
\eqref{eq:fp-expr} for the coefficients $\fp_{n,k}$ as well as the values
$\fp_{n,0}$ and $\fp_{n,\dd_n}$ follow from Theorems~\ref{thm:as-palin-2},
\ref{thm:fp-coeff}, and \ref{thm:ag-recur}. We also have
$\AS_{n,2k}(x) \in \CC[[x]]$ by \eqref{eq:as-gen}.

For $k > \dd_n$, Theorem~\ref{thm:as-palin-2} shows that
$\AS_{n,2k}(x) \in \CC[x]$ is an anti-palindromic polynomial. Therefore,
$\coeff{x^k} \, \AS_{n,2k}(x) = 0$, which coincides with \mbox{$\fp_{n,k} = 0$}
by definition. Hence, relation~\eqref{eq:fp-main} is valid for all $k \geq 0$.
As a consequence, \eqref{eq:fp-expr} holds for $0 \leq k \leq n$.

The even case $n+1$ also follows from Theorem~\ref{thm:fp-coeff}. This provides
the values $\fp_{n+1,k}$ for \mbox{$k \in \set{0,\dd_{n+1}}$}, while
$\fp_{n+1,k} = 0$ for \mbox{$k > \dd_{n+1}$} by definition. It remains to
evaluate the coefficients $\fp_{n+1,k}$ for \mbox{$1 \leq k < \dd_{n+1}$}.
Theorem~\ref{thm:fp-recur} shows that $\fh_{n+1,k} = (n+1) \, \fh_{n,k-1}$,
which implies the result. This completes the proof.
\end{proof}

\renewcommand\arraystretch{1.3}

As a consequence of Theorem~\ref{thm:fp-main}, we obtain for odd \mbox{$n \geq 1$}
and \mbox{$k \geq (n+1)/2$} the following implications in view of symmetry relations:
\[
\begin{array}{c}
  \AP_n(x) \text{\ satisfies\ } \eqref{eq:ap-ref} \\
  \Downarrow \\
  \AS_{n,2k}(x) \text{\ is anti-palindromic} \\
  \Downarrow \\
  \coeff{x^k} \, \AS_{n,2k}(x) = 0 \\
  \Downarrow \\
  \AG_{n,2k}^{(k)}(1) = 0 \\
  \Downarrow \\
  \fp_{n,k} = 0.
\end{array}
\]

Together with Corollary~\ref{cor:ac-as},
the symmetry properties imply the following recurrences.

\begin{corollary} \label{cor:recur}
Assume property \eqref{eq:ap-ref}. Let $n \geq 1$ be odd.
For $(n+1)/2 \leq k \leq n$, there are the following recurrences:
\begin{alignat*}{2}
  &\sum_{\nu=0}^{k} \binom{2k-n}{k-\nu} \ac_{n,\nu} &&= 0, \\
  &\sum_{\nu=0}^{k} \binom{2k-n}{k-\nu} \binom{n}{\nu} \alpha_{n-\nu,\nu} &&= 0, \\
  &\sum_{\nu=0}^{k} \binom{2k-\nu}{k} \binom{n}{\nu} \alpha_{n-\nu} &&= 0.
\end{alignat*}
\end{corollary}

For instance, the case $n=1$ of Corollary~\ref{cor:recur} provides that
$\ac_{1,0}+\ac_{1,1} = \alpha_{1,0} + \alpha_{0,1} = 0$,
which coincides with Corollary~\ref{cor:ac-recip},
and $2\alpha_1 + \alpha_0 = 0$ reflects \eqref{eq:ap-ref-4}.
Recall that $\EN_n(0) = \GN_{n+1}/(n+1)$, so Corollary~\ref{cor:recur}
applied to the Euler polynomials induces the following result for the
Genocchi numbers.

\begin{example}
Let $n \geq 1$ be odd. For $(n+1)/2 \leq k \leq n$, we have
\[
  \sum_{\nu=0}^{k} \binom{2k-\nu}{k} \binom{n}{\nu} \frac{\GN_{n+1-\nu}}{n+1-\nu} = 0.
\]
\end{example}

As a last application of Theorem~\ref{thm:fp-main}, we obtain the following
result fairly quickly compared to the usual approaches.

\begin{example} \label{exp:appl}
The special case of the Bernoulli polynomials $\BN_n(x)$ and the classical
Faulhaber polynomials in the context of the power-sum function
\[
   S_n(m) = \sum_{\nu=0}^{m-1} \nu^n = \frac{1}{n+1}( \BN_{n+1}(m) - \BN_{n+1} )
     \quad (m, n \geq 0)
\]
is comprehensively discussed in \cite{Kellner:2021}. We briefly show for odd
$n \geq 1$ that all parts fit together in a \emph{miraculous} way. There exists
a unique equation with the classical Faulhaber polynomial $\FF_n(y)$ such that
\[
  S_n(m) = \FF_n( y ) \quad \text{with\ } y = S_1(m) = \binom{m}{2},
\]
cf.~Knuth~\cite{Knuth:1993}. Using Theorem~\ref{thm:fp-main} with
$\AP_n(x) = \BN_n(x)$, it follows at once that
\[
  S_n(m) = \int_{0}^{m} \BN_n(x) \, dx = \int_{0}^{m} u' \, \FP_n(u) \, dx
    = \int_{0}^{2y} \FP_n(u) \, du = \FF_n( y ),
\]
where the coefficients of $\FF_n( y )$ depend on the coefficients $\fp_{n,k}$,
hence being expressible by terms of generalized reciprocal Bernoulli polynomials
(for explicit formulas see \cite{Kellner:2021}).
\end{example}


\section*{Acknowledgment}

We thank the anonymous referee for several suggestions.


\bibliographystyle{amsplain}

\end{document}